\documentclass{article}%
\usepackage{amsfonts,color}
\usepackage{amsmath}
\usepackage{amssymb}
\usepackage{graphicx}%
\providecommand{\U}[1]{\protect\rule{.1in}{.1in}}
\newtheorem{theorem}{Theorem}

\newtheorem{corollary}[theorem]{Corollary}

\newtheorem{definition}[theorem]{Definition}
\newtheorem{example}[theorem]{Example}

\newtheorem{lemma}[theorem]{Lemma}

\newenvironment{proof}[1][Proof]{\noindent\textbf{#1.} }{\ \rule{0.5em}{0.5em}}

\title{ Improving integrability via absolute summability: a general version of Diestel's Theorem}
\author{{\bf D. Pellegrino} \\ {\small Departamento de Matem\'atica}\\  {\small Universidade
Federal da Para\'iba} \\  {\small 58.051-900, Jo\~ao Pessoa, Brazil} \\  {\small e-mail: dmpellegrino@gmail.com}\\ {\bf P. Rueda }\\   {\small  Departamento de An\'alisis Matem\'atico} \\
 {\small  Universidad de Valencia}\\ {\small 46100 Burjassot - Valencia, Spain}\\   {\small e-mail: pilar.rueda@uv.es}\\  {\bf E.A. S\'anchez-P\'erez} \\   {\small Instituto Universitario de
Matem\'{a}tica Pura y Aplicada}\\   {\small Universitat Polit\`ecnica de Val\`encia} \\  {\small Camino
de Vera s/n, 46022 Valencia, Spain} \\ {\small  e-mail: easancpe@mat.upv.es}}
\date{}

\begin{document}

\maketitle


\begin{abstract}
A classical result by J. Diestel establishes that the composition  of a summing operator with a (strongly measurable) Pettis integrable function gives a Bochner integrable function. In this paper we show that a much more general result is possible regarding the improvement of the integrability of vector valued functions by the summability of the operator. After proving a general result, we center our attention in the particular  case given by the $(p,\sigma)$-absolutely continuous operators, that allows to prove a lot of special results on integration improvement for selected cases of classical Banach spaces ---including $C(K)$, $L^p$ and Hilbert spaces--- and operators ---$p$-summing, $(q,p)$-summing and $p$-approximable operators---.
\end{abstract}

\section{Introduction}

Let $X$ and $Y$ be Banach spaces.
Let $(\Omega,\Sigma,\mu)$ be a finite, positive, non-atomic measure space, and let $1\leq p<\infty$. Consider the space ${\mathcal P}_p(\mu,X)$  of all strongly measurable Pettis  $p$-integrable functions with respect to $\mu$, $f:\Omega\to X$ with the norm
$$
\|f\|_{\mathcal P_p}:=\sup\Big\{ \Big(\int |x'(f(w))|^p\, d\mu(w)\Big)^{1/p}:x'\in X^*,\|x'\|\leq 1\Big\},
$$
and let ${\mathcal B}_p(\mu,X)$ be the space of all Bochner $p$-integrable functions with respect to $\mu$, $f:\Omega\to X$ with the norm
$$
\|f\|_{\mathcal B_p}:= \Big(\int \|f(w)\|^p\, d\mu(w)\Big)^{1/p}.
$$

A linear operator $T\colon X\longrightarrow Y$ is absolutely $p$-summing if $(T(x_n))_{n=1}^\infty$ is absolutely $p$-summable in $Y$ whenever $(x_n)_{n=1}^\infty$ is weakly $p$-summable in $X$.
Diestel  proved in 1972 that an absolutely $1$-summing (from now on simply called absolutely summing) operator not only improves the summability of sequences but also the integrability of functions in the following sense (see \cite{Di}): \textit{ Consider $u\in {\mathcal L}(X;Y)$ and let $\tilde u:{\mathcal P}_1(\mu,X)\to {\mathcal P}_1(\mu,Y)$ be given by $(\tilde uf)(w):=u(f(w))$, $w\in \Omega$ and $f\in {\mathcal P}_1(\mu,X)$. Then $\tilde u$ belongs to ${\mathcal L}({\mathcal P}_1(\mu,X); {\mathcal B}_1(\mu,Y))$ if and only if  $u$ is absolutely summing.}

The transcendence of absolutely summing operators and their connection, provided by Diestel, with integrability has produced the appearance of several works dealing with improvements of integrability of functions via composition with absolutely summing operators. For instance,  Rodr\'{\i}guez \cite{Ro} has proved that if $u$ is absolutely summing then $\tilde u$ maps Dunford integrable functions to scalarly equivalent functions to Bochner integrable ones. Furthermore, he has proved that $U$ maps Birkhoff integrable or McShane integrable functions to Bochner integrable functions (see \cite{Ro} for definitions and for some other related works). Our aim  is to show that a variant of this improvement is fulfilled by  $R,S$-abstract summing operators in the sense of \cite{BoPeRu} (see also \cite{PeSa}). This  general result has many consequences. We will center our attention in a particular case, that contains  the classical result by Diestel and let to consider it  as a starting point of a scale of results concerning the relation between $(p,q)$-summability of $u$ and the improvement of integrability provided by $\tilde u$. In order to do it, our main tool is the study of the improvement of integrability given by operators belonging to a particular class of interpolated operator ideals of summing operators: the so called ideals of $(p,\sigma)$-absolutely continuous operators (see \cite{LS93,LS00,Matt87}). We will need to introduce also a new definition of integrability of strongly measurable (vector valued) functions that is given by some interpolation formula between the Bochner integral and the Pettis integral.

However, we  establish first  our main result ---Theorem \ref{prinint}--- in more abstract terms. The reason is that several recent papers have shown that, in fact, the main arguments that prove the important results on ideals of operators improving summability, can be formulated and proved in a very general setting, and all the properties that hold for the classical case work also in this abstract framework. This also happens in the case of the problem analyzed in the present paper. In  \cite{BoPeRu}, the definition of RS-abstract p-summing operator is given, and a generalized Pietsch's Domination Theorem is proved for it. In fact, it is proved that all known Pietsch type theorems are particular cases of this one, including the multilinear cases. For this reason, this is the definition that we adopt. In \cite{PeSa}, an even more general definition of $RS$-summing operator is given removing two hypothesis and still preserving the Domination Theorem. In this direction, a general framework regarding all these arguments is presented in \cite{PeSaSe}, and it is shown how can a lot of results of this kind be considered as particular instances of a general principle; the interested reader can find there ---and in the references in this paper---, a general explanation of the general topics and techniques in operator theory that are covered with the RS-abstract summing operators.

In the second part of the paper (Section \ref{4}), we present a lot of examples and applications of our results for particular sets of classical Banach spaces and operators that are better known than the $(p,\sigma)$-absolutely continuous operators. We will follow mainly the results in \cite{Matt87}, and also in \cite{LS93,LS97}.

\section{Basic concepts and notation}

We will use standard Banach space notation. If $1 \le p \le \infty$, we will write $p'$ for the extended real number that satisfies that $1/p+1/p'=1$. Unless stated otherwise $X$, $Y$, $E$ and $F$ will be Banach spaces. We will write $X^*$ for the dual of $X$ and $B_{X^*}$ denotes the closed unit ball of $X^*$ endowed with the weak star topology. Our fundamental references for the general theory of summing operators and general operator ideals are
\cite{DF92,DJT95,Pie80}.  The definition of $RS$-abstract summing operator is the following.

Consider  arbitrary sets $X$ and $Y$
and
$E$, let $\mathcal H$
be a family of mappings from
$X$
to
$Y$,
$G$
be a Banach space and
$K$
be a compact
Hausdorff topological space.

Let $R: K \times E \times G \to [0,\infty)$ and $S:\mathcal H \times E \times G \to [0,\infty)$ be mappings such that
the mapping $R_{x,b}:K \to [0,\infty)$, $R_{x,b}(\varphi)=R(\varphi,x,b)$ is continuous for each $x \in E$ and $b \in G.$


Using these elements, the following notion is introduced in  \cite[Def.2.1]{BoPeRu} and \cite{PeSa}.

\begin{definition}
Let $R$ and $S$ as above and $0 < p < \infty.$ A mapping $f \in \mathcal H$ is said to be $RS$-abstract $p$-summing if there is a constant $C_1 >0$ such that
$$
\Big( \sum_{j=1}^m S(f,x_j,b_j)^p \Big)^{\frac{1}{p}} \le C_1 \sup_{\varphi \in K}
\Big( \sum_{j=1}^m R(\varphi,x_j,b_j)^p \Big)^{\frac{1}{p}},
$$
for all $x_1,...,x_n \in E$, $b_1,...,b_m \in G$ and $m \in \mathbb N$.
\end{definition}

 A particular instance of this abstract notion of summability is given by the interpolated class of linear operators cited above. Let $1 \le p < \infty$ and $0 \le \sigma < 1$. A linear operator $T:X \to Y$ is said to be $(p, \sigma)$-absolutely continuous if for every $x_1,...,x_n \in X$,
$$
\Big( \sum_{i=1}^n \big\| T(x_i) \big\|^{\frac{p}{1-\sigma}} \Big)^{\frac{1-\sigma}{p}} \le \sup_{x' \in B_{X^*}} \Big( \sum_{i=1}^n \big( | \langle x_i, x' \rangle|^{1-\sigma} \| x_i \big\|^\sigma \big)^{\frac{p}{1-\sigma}} \Big)^{\frac{1-\sigma}{p}}.
$$
The space of all linear  operators from $X$ to $Y$  satisfying this property will be denoted in this paper by $\Pi_p^\sigma(X,Y)$. It is a normed space, being $\pi_p^{\sigma}(T)$ the norm computed as the infimum of all the constants $C>0$ in the inequality above (see  \cite{LS93,LS00,Matt87}). It is easy to see that 
\begin{equation}
\label{000}
\Pi_{p_1}^{\sigma_1} \subseteq \Pi_{p_2}^{\sigma_2}
\end{equation}

for $1 \le p_1 \le p_2 < \infty$ and $0 \le \sigma_1 \le \sigma_2 < 1$ (see \cite[Prop. 3.3]{Matt87}; see also \cite{LS93}).

The usual space of $(p,q)$-summing operators will be relevant in this paper too, for $1 \le q \le p < \infty$. This is given by all linear operators $T:X \to Y$ for which  there exists a constant $C>0$ depending on $T$ such that  for every $x_1,...,x_n \in X$,
$$
\Big( \sum_{i=1}^n \big\| T(x_i) \big\|^{p} \Big)^{\frac{1}{p}} \le C\sup_{x' \in B_{X^*}} \Big( \sum_{i=1}^n | \langle x_i, x' \rangle|^{{q}} \Big)^{\frac{1}{q}}.
$$
We will use the symbol $\Pi_{p,q}(X,Y)$ to denote the component of this operator ideal when the linear operators from $X$ to $Y$ are considered, and the norm for such an operator $T$ will be denoted by $\pi_{p,q}(T)$  and defined as the infimum of all constants $C$ as above. For $p=q$, we get the operator ideal of $p$-summing operators; the notation will be in this case $\Pi_p(X,Y)$ and $\pi_p(T)$.

The following well-known result will be used several times in this paper. We write the proof for the aim of completeness; it can be found in \cite[Prop.4.2]{Matt87}.

\begin{lemma} \label{le}
Let $1 \le p < \infty$ and $0 \le \sigma <1$. Then $\Pi_p \subseteq \Pi_{\frac{p}{1 - \sigma}} \subseteq \Pi_{p}^{\sigma} \subseteq \Pi_{\frac{p}{1-\sigma},p}$.
\end{lemma}
\begin{proof}
The first inclusion is well-known. For the other ones, first note that the strong summable parts ---the left hand sides--- of the inequalities defining the three operator ideals appearing in the  two other  inclusions are the same. Thus, the following inequalities for finite sets of vectors $x_1,...,x_n$ of a Banach space $X$ give the result.
$$
\sup_{x' \in B_{X^*}} \Big( \sum_{i=1}^n | \langle x_i, x' \rangle|^{\frac{p}{1-\sigma}}  \Big)^{\frac{1-\sigma}{p}} \le
\sup_{x' \in B_{X^*}} \Big( \sum_{i=1}^n  \big( |\langle x_i, x' \rangle|^{1-\sigma} \|x_i\|^{\sigma} \big)^{\frac{p}{1-\sigma}} \Big)^{\frac{1-\sigma}{p}}
$$
$$
= \sup_{x' \in B_{X^*}} \Big( \sum_{i=1}^n   |\langle x_i, x' \rangle|^{p} \|x_i\|^{\frac{p \sigma}{1-\sigma}} \Big)^{\frac{1-\sigma}{p}}
\le \max_{i=1,...,n} \|x_i\|^{\sigma} \cdot \sup_{x' \in B_{X^*}}  \big( \sum_{i=1}^n   |\langle x_i, x' \rangle|^{p} \big)^{\frac{1-\sigma}{p}}
$$
$$
\le \sup_{x' \in B_{X^*}}  \Big( \sum_{i=1}^n   |\langle x_i, x' \rangle|^{p} \Big)^{\frac{1}{p}} .
$$
\end{proof}

\section{A characterization of $RS$-abstract $p$-summing operators in terms of integrability.}

Let $K$ be a compact set and $H$ a space of operators. Let $E$ and $G$ be Banach spaces.
Let $S: H \times E \times G \to \mathbb R^+$ and $R: K \times E \times G \to \mathbb R^+$ be functions such that
\begin{equation}\label{cero}
S(u,0,b)=S(v,x,0)=0=R(\phi,0,b)=R(\varphi,x,0)
\end{equation}
for all $u,v\in H$, $x\in E$, $b\in G$, $\phi,\varphi \in K$.
 Let $(\Omega,\Sigma,\mu)$ be a non-atomic finite measure space, and let $M(\Omega,E)$ and $N(\Omega,G)$, $i=1,2$, be sets of strongly measurable (classes of $\mu$-a.e. equal) functions with values in $E$ and $G$ respectively. Consider subsets $M\subset M(\Omega,E)$ and $N\subset N(\Omega,G)$ that contain the simple functions. We define the real functions $\tilde S: H \times M\times N \times \Omega\to \mathbb R^+$ and $\tilde R: K \times M\times N \times \Omega \to \mathbb R^+$ given by
$$\tilde S(u,f,g,w):=S(u, f(w), g(w))$$ 
and 
$$\tilde R(\varphi,f,g,w):=R(\varphi,f(w),g(w)). $$

As our objective is to show how $RS$-abstract summing operators improve the integrability of functions, we will assume that the function $\tilde R_{\varphi, f,g}(w):=\tilde R(\varphi,f,g,w)$ is $\mu$-integrable for all $\varphi\in K$, $f\in M$ and $g\in N$. Then, under this assumption, our aim is to relate summability of $u$ with integrability of $\tilde S_{f,g}(w):=\tilde S(u, f,g, w)$, $f\in M$ and $g\in N$.




Notice that the definition of $u$ being $RS$-abstract $p$-summing is equivalent to the fact that there is $C>0$ such that for finite families
$x_1,...,x_n \in E$, $b_1,...,b_n \in G$ and  $a_{1,1},...,a_{1,n},...,a_{n,1},...,a_{n,n} \in \mathbb R^+$,
$$
\Big( \sum_{i,j=1}^n a_{i,j} S(u,x_i,b_j)^p \Big)^{1/p} \le C
\sup_{\varphi \in K} \Big( \sum_{i,j=1}^n a_{i,j} R(\varphi, x_i, b_j)^p \Big)^{1/p}.
$$

\begin{theorem} \label{prinint}
An operator $u:X \to Y$ is $RS$-abstract summing if and only if there is a constant $C>0$ such that
$$
\int_\Omega \tilde S(u,f,g,w)\, d\mu(w)\leq C\sup_{\varphi\in K}\int_\Omega \tilde R(\varphi,f,g,w)\, d\mu(w)
$$
 for all simple functions $f\in M$ and all simple functions $g\in N$.

\end{theorem}
\begin{proof}
($\Rightarrow$)
Take a pair of simple functions $f= \sum_{i=1}^n x_i \chi_{A_i} \in M$ and $g= \sum_{i=1}^m b_j \chi_{B_j} \in N$, with $x_1,\ldots, x_n\in E$, $b_1,\ldots, b_m\in G$, $A_1,\ldots, A_n,B_1,\ldots, B_m\in \Sigma$. Then, since we are assuming both  $u$ is  $R,S$-abstract summing and  condition~\eqref{cero},
\begin{eqnarray*}
\int_\Omega \tilde S(u,f,g,w) \,d \mu(w) &=& \int_\Omega S(u,f(w),g(w))\,  d \mu(w) \\
&=& \sum_{i=1}^n \sum_{j=1}^m  S(u,x_i,b_j) \mu(A_i \cap B_j) \\
&\le& C \sup_{\varphi \in K} \sum_{i=1}^n \sum_{j=1}^m  R(\varphi,x_i,b_j) \mu(A_i \cap B_j)\\
&=& C\sup_{\varphi \in K} \int_\Omega R(\varphi, \sum_{i=1}^n x_i \chi_{A_i}, \sum_{j=1}^m b_j \chi_{B_j} ) \,d \mu(w)\\
&=& C \sup_{\varphi \in K} \int_\Omega \tilde R(\varphi, f, g, w)\, d \mu(w).\\
\end{eqnarray*}

($\Leftarrow$)  Consider $x_1,...,x_n \in E$ and $b_1,...,b_n \in G$. Since the measure is non-atomic, there are pairwise disjoint measurable sets $A_1,...,A_n$ in $\Sigma$ such that $\mu(A_1)=...=\mu(A_n)=: \alpha >0$.
Then
\begin{eqnarray*}
\alpha \sum_{i=1}^n S(u,x_i,b_i) &= &\sum_{i=1}^n \mu(A_i) S(u, x_i, b_i)
\ = \int_\Omega \sum_{i=1}^n S(u, x_i,b_i) \chi_{A_i}(w)\,  d \mu(w)\\
&=& \int_\Omega S(u, \sum_{i=1}^n x_i \chi_{A_i}(w), \sum_{i=1}^n b_i \chi_{A_i}(w)) \,d \mu(w)\\
&=&  \int_\Omega \tilde S(u, \sum_{i=1}^n x_i \chi_{A_i}, \sum_{i=1}^n b_i \chi_{A_i},w) \, d \mu(w)\\
&\le& C \sup_{\varphi \in K} \int_\Omega \tilde R(u, \sum_{i=1}^n x_i \chi_{A_i}, \sum_{i=1}^n b_i \chi_{A_i},w) \, d \mu(w)\\
&=&C \sup_{\varphi \in K} \int_\Omega \sum_{i=1}^n R(\varphi, x_i, b_i) \chi_{A_i}(w) \, d\mu(w) \\
&=&C \sup_{\varphi \in K}  \sum_{i=1}^n R(\varphi, x_i, b_i) \mu({A_i} ) \ = \ \alpha C \sum_{i=1} R(\varphi,x_i,b_i).\\
\end{eqnarray*}
The proof is done.
\end{proof}

\medskip
One of the main examples of Theorem \ref{prinint} is  Diestel's result \cite{Di}, where $u:X  \to Y$ is  absolutely summing if and only if its associated composition operator carries strongly measurable  Pettis integrable functions to Bochner integrable functions. It can be easily seen that it is a particular case of our theorem, once we consider the completion of simple functions to the space of Pettis integrable functions.

\begin{corollary}\label{Diestel}
Let $X$ and $Y$ be Banach spaces, and let $u:X\to Y$ be a continuous linear operator.  Then $u$ is absolutely summing if and only if $\tilde u:{\mathcal P}_1(\mu;X)\to {\mathcal B}_1(\mu;Y)$ is well-defined and continuous.
\end{corollary}

\begin{proof}
  Consider $E =X$, $K=B_{X^{*}}$
and $G=\mathbb{K}$ (the scalar field). Take $\mathcal{H}=\mathcal L(X,Y)$ the space of all continuous linear operators from $X$
into $Y$ and define $R$ and $S$ by:
\[
R\colon B_{X^{*}}\times X\times\mathbb{K} \longrightarrow[0,\infty) ~,~
R(\varphi,x,\lambda)=\vert\lambda\vert\vert\varphi(x)\vert
\]
\[
S\colon \mathcal L(X,Y)\times X\times\mathbb{K} \longrightarrow[0,\infty) ~,~
S(T,x,\lambda)=\vert\lambda\vert\Vert T(x)\Vert.
\]
With $R$ and $S$ so defined, a linear operator $u\colon X\longrightarrow Y$ is $RS$-abstract $1$-summing if and only if it is absolutely summing.
The fact that $\tilde u:{\mathcal P}_1(\mu;X)\to {\mathcal P}_1(\mu;Y)$ is continuous and the density of the simple functions in ${\mathcal P}_1(\mu;X)$ gives the result. The reader can find this argument written in a more precise way in the proof of the forthcoming Theorem \ref{sigma}.
\end{proof}

\section{The case of $(1,\sigma)$-absolutely continuous operators}

This section is devoted to  $(1,\sigma)$-absolutely continuous operators. Our aim is to show the power of our main result ---Theorem \ref{prinint}--- by showing that the classical result by Diestel (Corollary \ref{Diestel}) is in fact an extreme case of a chain of results that characterize integrability of the functions in terms of the summability properties of the operator $u$. The idea is that absolutely summing operators are the starting point (for $\sigma=0$) of a series ---ordered by inclusion--- of classes of linear operators, whose ending point (for $\sigma=1$) is the class of all continuous linear operators. The class that holds the position $\sigma$, for $0\leq \sigma\leq 1$, of this chain is the class of $(1,\sigma)$-absolutely continuous operators. According to this interpolating idea, we introduce the space ${\mathcal P}_p^\sigma(\mu;X)$, that can be considered as the interpolated class, for $0\leq \sigma\leq 1$ and $1 \le p < \infty$, between ${\mathcal P}_p(\mu;X)$ and ${\mathcal B}_p(\mu;X)$.

Our aim is to show that an operator $u:X \to Y$ is $(1,\sigma)$-absolutely continuous if and only if its associated composition operator $\tilde u:{\mathcal P}_1(\mu;X)\to {\mathcal P}_1(\mu;Y)$ carries $(1/(1-\sigma),\sigma)$-Pettis integrable functions to $1/(1-\sigma)$-Bochner integrable functions. Several new results concerning the improvement of integrability via summability of operators are  derived.

We need first some definitions. Let $(\Omega,\Sigma,\mu)$ be a non-atomic finite measure space. Consider $0 \le \sigma \le 1$ and $1 \le p < \infty$, and the space ${\mathcal S}_{p}^\sigma(\mu,X)$ of all equivalence classes with respect to $\mu$ of simple  functions with values in the Banach space $X$. Since all of  the functions in it are simple, they satisfy that
$$
 \big( | \langle f(w), x' \rangle |^{1-\sigma} \|f(w)\|^\sigma \big)^p \in L^1(\mu)
$$
for all ${x' \in X^*},$ and
$$
\Phi_{p,\sigma}(f):= \sup_{x' \in B_{X^*}}  \Big( \int \big( | \langle f(w), x' \rangle |^{1-\sigma} \|f(w)\|^\sigma \big)^p \, d \mu  \Big)^{1/p} < \infty.
$$
A seminorm for this space can be given by the convexification $\|\cdot \|_{p,\sigma}$ of $\Phi_{p,\sigma}$ defined by
$$
\|f\|_{p,\sigma}:= \inf \big\{ \sum_{i=1}^n \Phi_{p,\sigma}(f_i): f=\sum_{i=1}^n f_i \big\},
$$
for $f\in {\mathcal S}_p^\sigma(\mu,X)$ .
Since for all $p$ and $\sigma,$ $\|\cdot\|_{\mathcal{P}_p} \le \Phi_{p,\sigma}(\cdot)$, we have that $\|\cdot\|_{p,\sigma}$ is in fact a norm, and ${\mathcal S}_p^\sigma(\mu,X) \subseteq {\mathcal P}_{p}(\mu,X)$ continuously for all $0 \le \sigma \le 1$.

Write $\overline{S_p^\sigma(\mu;X)}$ for the completion of $(S_p^\sigma(\mu;X), \|\cdot \|_{p,\sigma})$.
Define the space ${\mathcal P}_p^\sigma(\mu,X)$ by
$$
{\mathcal P}_p^\sigma(\mu,X):=\overline{S_p^\sigma(\mu;X)}\cap {\mathcal P}_p(\mu;X),
$$
endowed with the norm induced by $\overline{S_p^\sigma(\mu;X)}$.

For the applications,  the case $p=1$ will be relevant.
We will use the fact that given a continuous linear operator $u:X\to Y$, then $\tilde u:{\mathcal P}_1(\mu;X) \to {\mathcal P}_1(\mu;Y) $ given by $\tilde u(f)=u\circ f$, is well-defined and continuous (see the proof of the theorem in \cite{Di}).


\begin{theorem}\label{sigma}
Let $0 \le \sigma \le 1$.
An operator  $u:X \to Y$ is $(1,\sigma)$-absolutely continuous if and only if the composition operator $\tilde{u}:{\mathcal P}_{1/(1-\sigma)}^\sigma(\mu,X)\to {\mathcal B}_{1/(1-\sigma)}(\mu,Y)$ given by $\tilde u(f):=u\circ f$ is well defined and continuous.
\end{theorem}
\begin{proof} Let us see first that $(1,\sigma)$-absolutely continuous operators are $R,S$-abstract $1$-summing for suitable $R$ and $S$.
Take the functions
$S: \mathcal{L}(X,Y) \times X \times \mathbb R \to \mathbb R^+$ given by
$S(u,x,a):=\|u(x)\|^{1/(1 - \sigma)}|a|$,
 and $R: B_{X^*} \times X \times  \mathbb R \to \mathbb R^+$ by
 $ R(x',x,b) := |\langle x, x' \rangle|^{} \|x\|^{\sigma/(1-\sigma)}|b|$.

For this $R$ and $S$, an operator $u:X \to Y$ is $RS$-abstract $1$-summing if and only if there is a constant $C>0$ such that for each pair of finite sets $x_1,...,x_n \in X$ and $b_1,...,b_n \in \mathbb R$,
\begin{eqnarray*}
\sum_{i=1}^n \|u(x_i)\|^{1/(1-\sigma)}|b_i|&=& \sum_{i=1}^n S(u,x_i,b_i)  \\
& \le&  C \sup_{x' \in B_{X^*}} \sum_{i=1}^n R(x',x_i,b_i) \\
&=& C \sup_{x' \in B_{X^*}} \sum_{i=1}^n |\langle x_i, x' \rangle|^{} \|x_i\|^{\sigma/(1-\sigma)}|b_i|,\\
\end{eqnarray*}
 that is, if the operator is $(1,\sigma)$-absolutely continuous.

The second step of the proof consists of defining  maps $\tilde R$ and $\tilde S$ in order to apply Theorem \ref{prinint}.
  Take  ${\mathcal M}(\mu;X)$ be the set of all strongly measurable functions with values in $X$.
 Define now
$\tilde S: \mathcal{L}(X,Y) \times {\mathcal M}(\mu,X) \times \{ \chi_{\Omega} \} \times \Omega \to \mathbb R^+$ given by
$$
\tilde S(u,f,\chi_\Omega, \omega):= S(u,f(\omega),\chi_\Omega(\omega)) =\|u \circ f(\omega)\|^{1/(1 - \sigma)},
$$
 and
 $
 \tilde R: B_{X^*} \times {\mathcal M}(\mu,X) \times  \{ \chi_{\Omega} \} \times \Omega \to \mathbb R^+
 $
  by
 $$
 \tilde R(x',f,\chi_\Omega, \omega) := R(x',f(\omega),\chi_\Omega(\omega))= |\langle f(\omega), x' \rangle|^{} \|f\|^{\sigma/(1-\sigma)}.
 $$

An application of Theorem \ref{prinint} shows that $u$ is $(1,\sigma
)$-absolutely continuous if, and only if,
\begin{equation}
\int_{\Omega}\Vert u\circ f(w)\Vert^{1/(1-\sigma)}\,d\mu(w)\leq C\sup
_{x^{\prime}\in B_{X^{\ast}}}\int_{\Omega}\left\vert \left\langle
f(w),x^{\prime}\right\rangle \right\vert \left\Vert f(w)\right\Vert
^{\sigma/\left(  1-\sigma\right)  }\,d\mu(w),\label{q11}%
\end{equation}
for all simple function $f\in{\mathcal{M}}(\mu;X)$. This proves that $u$ is
$(1,\sigma)$-absolutely continuous if, and only if, $u_{0}:{\mathcal{S}%
}_{1/(1-\sigma)}^{\sigma}(\mu,X)\rightarrow{\mathcal{B}}_{1/(1-\sigma)}%
(\mu;Y),$ given by $u_{0}(f):=u\circ f$, is continuous. In fact, from
(\ref{q11}) we have%

\begin{align*}
& \left(  \int_{\Omega}\Vert u\circ f(w)\Vert^{1/(1-\sigma)}\,d\mu(w)\right)
^{1-\sigma}\\
&\leq C^{1-\sigma}\left(  \sup_{x^{\prime}\in B_{X^{\ast}}}%
\int_{\Omega}\left\vert \left\langle f(w),x^{\prime}\right\rangle \right\vert
\left\Vert f(w)\right\Vert ^{\sigma/\left(  1-\sigma\right)  }\,d\mu
(w)\right)  ^{1-\sigma}\\
\end{align*}
and since the expression in the left is a norm, we also have for $f=%
{\textstyle\sum_{i}}
f_{i}$,%
\begin{align*}
&  \left(  \int_{\Omega}\Vert u\circ%
{\textstyle\sum_{i}}
f_{i}(w)\Vert^{1/(1-\sigma)}\,d\mu(w)\right)  ^{1-\sigma}\\
&  \leq%
{\textstyle\sum_{i}}
\left(  \int_{\Omega}\Vert u\circ f_{i}(w)\Vert^{1/(1-\sigma)}\,d\mu
(w)\right)  ^{1-\sigma}\\
&  \leq%
{\textstyle\sum_{i}}
C^{1-\sigma}\left(  \sup_{x^{\prime}\in B_{X^{\ast}}}\int_{\Omega}\left\vert
\left\langle f_{i}(w),x^{\prime}\right\rangle \right\vert \left\Vert
f_{i}(w)\right\Vert ^{\sigma/\left(  1-\sigma\right)  }\,d\mu(w)\right)
^{1-\sigma},
\end{align*}
and thus%
\begin{align*}
&  \left(  \int_{\Omega}\Vert u\circ f(w)\Vert^{1/(1-\sigma)}\,d\mu(w)\right)
^{1-\sigma}\\
&  \leq C^{1-\sigma}\inf_{f=%
{\textstyle\sum_{i}}
f_{i}}%
{\textstyle\sum_{i}}
\left(  \sup_{x^{\prime}\in B_{X^{\ast}}}\int_{\Omega}\left\vert \left\langle
f_{i}(w),x^{\prime}\right\rangle \right\vert \left\Vert f_{i}(w)\right\Vert
^{\sigma/\left(  1-\sigma\right)  }\,d\mu(w)\right)  ^{1-\sigma}\\
&  =C^{1-\sigma}\left\Vert f\right\Vert _{p,\sigma}.
\end{align*}

Note that $u_0(f)=\tilde u(f)$ for all $f$ in ${\mathcal S}_{1/(1-\sigma)}^\sigma(\mu,X)$. The last step of the proof consists of proving that the composition operator  ${u}_0:{\mathcal S}_{1/(1-\sigma)}^\sigma(\mu,X)\to {\mathcal B}_{1/(1-\sigma)}(\mu,Y)$ can be extended to the whole space ${\mathcal P}_{1/(1-\sigma)}^\sigma(\mu,X)$, i.e. that the operator    $\tilde{u}:{\mathcal P}_{1/(1-\sigma)}^\sigma(\mu,X)\to {\mathcal B}_{1/(1-\sigma)}(\mu,Y)$ is well-defined and continuous. This is a direct consequence of the following argument.
The operator  ${u}_0:{\mathcal S}_{1/(1-\sigma)}^\sigma(\mu,X)\to {\mathcal B}_{1/(1-\sigma)}(\mu,Y)$ can be extended by continuity to an operator $U: {\mathcal P}_{1/(1-\sigma)}^\sigma(\mu,X)\to {\mathcal B}_{1/(1-\sigma)}(\mu,Y)$. We have to show that $U(f)=\tilde u(f)$ for all  $f$  in ${\mathcal P}_{1/(1-\sigma)}^\sigma(\mu,X)$. Fix $f \in {\mathcal P}_{1/(1-\sigma)}^\sigma(\mu,X)$ and take a sequence $(f_n) \in {\mathcal S}_{1/(1-\sigma)}^\sigma(\mu,X)$ such that $f_n \to f$ in the norm $\|\cdot\|_{{\mathcal P}_{1/(1-\sigma)}^\sigma}$ of ${\mathcal P}_{1/(1-\sigma)}^\sigma(\mu,X)$.
Since ${\mathcal P}_{1/(1-\sigma)}^\sigma(\mu,X) \subseteq {\mathcal P}_{1}(\mu,X)$ continuously, the sequence $(f_n)$ converges to $f$ for the norm $\|\cdot \|_{{\mathcal P}_1}$. As $\tilde u:{\mathcal P}_1(\mu;X) \to {\mathcal P}_1(\mu;Y) $ is continuous,  $\tilde{u}(f_n)$ converges to the function $\tilde{u}(f) \in \mathcal{P}_1(\mu,Y)$. It follows that
$$
U(f)=\|\cdot\|_{{\mathcal P}_{1/(1-\sigma)}^\sigma}-\lim_n U(f_n)=\|\cdot\|_{{\mathcal P}_{1}}-\lim_n U(f_n)= \|\cdot\|_{{\mathcal P}_{1}}-\lim_n \tilde{u}(f_n) = \tilde{u}(f) ,
$$
and so, $\tilde{u}(f)$ is Bochner $1/(1-\sigma)$-integrable. This gives the result.
\end{proof}

Let us finish this section with two relevant examples.

\begin{example}
Let $K$ be a compact Hausdorff topological space, and consider the space $C(K)$  of all continuous functions defined on it endowed with the usual sup norm. If $1 \le p \le \infty$  and  $\nu$ is a measure,
we denote as usual by $L_p(\nu)$ the space of all (classes of equivalence of) $p$-integrable functions.
\begin{itemize}
\item[(1)] Assume that $\nu$ is a regular Borel probability measure on $K$ and consider the canonical $p$-summing identification map  $j_p:C(K)\to L_p(\nu)$. By Lemma \ref{le} the mapping $j_p$ is $(1,\sigma)$- absolutely continuous for any $0\leq \sigma\leq 1$. Let us show how $j_p$ improves integrability by a straightforward application of Theorem \ref{sigma}. Let $\sigma= 1/p'$ (that is,  $p=1/(1-\sigma)$) and consider a strongly measurable function  $f:\Omega\to C(K)$.  If $f$ belongs to ${\mathcal P}_p^\sigma(\mu;C(K))$ then, $j_p\circ f:\Omega \to L_p(\nu)$ belongs to ${\mathcal B}_p(\mu;L_p(\nu))$.

\item[(2)] Take $\nu$ any $\sigma$-finite measure. Let $1\leq r\leq 2$ and $2\leq
q<\infty$, and define $\sigma_{0}:=1/q^{\prime}$. Then $\tilde{u}%
:{\mathcal{P}}_{q}^{\sigma_{0}}(\mu;C(K))\rightarrow{\mathcal{B}}_{q}%
(\mu;L_{r}(\nu))$ is well-defined and continuous for any $u\in{\mathcal{L}%
}(C(K);L_{r}(\nu))$. In fact, from \cite[Th.3.5]{DJT95} we know that
$\mathcal{L}(C(K);L_{r}(\nu))=\Pi_{2}(C(K);L_{r}(\nu))$. Since $1/(1-\sigma
_{0})=q\geq 2$,  by Lemma
\ref{le}  we obtain
\[
\Pi_{2}(C(K);L_{r}(\nu))\subseteq \Pi_{q}(C(K);L_{r}(\nu))\subseteq\Pi
_{1}^{\sigma_0}(C(K);L_{r}(\nu))
\]

and the Theorem \ref{sigma} gives the result.

\end{itemize}

\end{example}

\section{Applications: $(p,q)$-summing operators and integrability} \label{4}

The following inclusion property will be  useful in this section. It follows easily from the definitions; note that the case $\sigma_2 =1$ gives the Bochner $p$-norm. Recall that we are considering  a finite and non-atomic measure space $(\Omega, \Sigma,\mu)$.

\begin{lemma} \label{leinc}
Let $1 \le p < \infty$ and $0 \le \sigma_1 \le \sigma_2  \le 1$.
For a simple function $f$, $\Phi_{p,\sigma_1}(f) \le \Phi_{p, \sigma_2}(f)$. Consequently, ${\mathcal P}_p^{\sigma_2}(\mu,X) \subseteq {\mathcal P}_p^{\sigma_1}(\mu,X)$ with continuous inclusion.
\end{lemma}
\begin{proof}
Since for every  ${x' \in B_{X^*}},$ we have that $| \langle f(w),x' \rangle | \le \|f(w)\|$, the inequality
$$
\Phi_{p,\sigma_1}(f)^p = \sup_{x' \in B_{X^*}}  \, \int \big( | \langle f(w), x' \rangle |^{1-\sigma_1} \|f(w)\|^{\sigma_1} \big)^p \, d \mu
$$
$$
 =
 \sup_{x' \in B_{X^*}}  \, \int \big( | \langle f(w), x' \rangle |^{1-\sigma_2} | \langle f(w), x' \rangle |^{\sigma_2-\sigma_1}  \|f(w)\|^{\sigma_1} \big)^p \, d \mu
$$
$$
\le  \sup_{x' \in B_{X^*}}  \, \int \big( | \langle f(w), x' \rangle |^{1-\sigma_2} \|f(w)\|^{\sigma_2-\sigma_1}  \|f(w)\|^{\sigma_1} \big)^p \, d \mu
$$
$$
= \sup_{x' \in B_{X^*}}  \, \int \big( | \langle f(w), x' \rangle |^{1-\sigma_2} \|f(w)\|^{\sigma_2} \big)^p \, d \mu = \Phi_{p,\sigma_2}(f)^p
$$
holds. The continuous inclusion in the statement  is a direct consequence of this inequality and the definitions.
\end{proof}



The following three results are direct applications of Theorem \ref{sigma} under the hypothesis of coincidences  between the operator ideals $\Pi_p$, $\Pi_1^\sigma$ and $\Pi_{(\frac{1}{1-\sigma},1)}$. They are the source of a lot of particular applications, that will be written in the next subsection.

\begin{corollary} \label{c1}
Let $1 \leq s<\infty$ and $\sigma=1/s'$. If $u\in \Pi_s(X;Y)$ then $\tilde u:{\mathcal P}_s^\sigma(\mu;X)\to {\mathcal B}_s(\mu;Y)$ is well-defined and continuous. In general, if $u\in \Pi_s(X;Y)$ and $1/s' \le \sigma_1,$  then $\tilde u:{\mathcal P}_s^{\sigma_1}(\mu;X)\to {\mathcal B}_s(\mu;Y)$ is well-defined and continuous.
\end{corollary}
\begin{proof} The case $s=1$ reduces to Diestel's theorem.
First note that $s=1/(1-\sigma)$. Then by Lemma \ref{le}, if $u \in \Pi_s(X,Y)$, we have that $u \in \Pi_1^\sigma (X,Y)$. By Theorem \ref{sigma} we have then that the operator
$$
\tilde{u}: {\mathcal P}_s^\sigma (\mu;X) = {\mathcal P}_{1/(1-\sigma)}^\sigma(\mu,X)\to {\mathcal B}_{1/(1-\sigma)}(\mu,Y) = {\mathcal B}_s(\mu;Y)
$$
is well-defined and continuous. The second statement is a consequence of Lemma \ref{leinc}, since ${\mathcal P}_s^{\sigma_1} (\mu;X) \subseteq {\mathcal P}_s^\sigma(\mu;X)$ continuously.
\end{proof}

The next results are obvious, we write them as corollaries for further use.

\begin{corollary} \label{c2}
Let $1\leq p<\infty$ and $\sigma=1/p'$. Assume that $\Pi_{(p;1)}(X;Y)=\Pi_1^\sigma (X,Y).$ Then $u\in \Pi_{(p;1)}(X;Y)$ if and only if  $\tilde u:{\mathcal P}_p^\sigma(\mu;X)\to {\mathcal B}_p(\mu;Y)$ is well-defined and continuous.
\end{corollary}

\begin{corollary} \label{c3}
Let $1\leq p\leq q<\infty$ and $\sigma=1/q'$. Assume that ${\mathcal L}(X;Y)=\Pi_{p}(X;Y).$ Then $\tilde u:{\mathcal P}_q^\sigma(\mu;X)\to {\mathcal B}_q(\mu;Y)$ is well-defined and continuous for any $u\in {\mathcal L}(X;Y)$.
\end{corollary}
\begin{proof}
It follows from the inclusions $\Pi_p(X;Y)\subset \Pi_{q}(X;Y)\subset \Pi_{1,\sigma}(X;Y)$ (Lemma \ref{le}) and  Theorem \ref{sigma}.
\end{proof}

In what follows we deal with classical Banach spaces. We will use some results on coincidence of some classical operator ideals with the interpolated class of the ideals of $(p,\sigma)$-absolutely continuous operators; our main source is the paper \cite{Matt87}. We will also use some classical results.

\subsection{Operators improving integrability of Hilbert space valued functions}

Consider a Hilbert space $L^2$. It is well-known (see \cite[Corollary 11.16]{DJT95}) that for each $1 \le p < \infty$, $\Pi_p(L^2,L^2)=\Pi_1(L^2,L^2)$. Let $1 \le p < \infty$ and let $\mathcal A_p(L^2,L^2)$ be the corresponding component of the ideal of the $p$-approximable operators. It is well-known that $\Pi_2(L^2,L^2)=\mathcal A_2(L^2,L^2)$ (see e.g. \cite[Theorem 20.5.1]{Jar} or \cite[Theorem 4.10]{DJT95}). In this setting, we can obtain the following

\begin{corollary}
Let $\mu$ be a non-atomic finite positive measure.
Let $1 \le p < \infty$ and $\sigma=1/p'$ and consider  $u:L^2 \to L^2$ a $2p$-approximable operator. Then
the composition operator $\tilde{u}:{\mathcal P}_p^\sigma(\mu,L^2)\to {\mathcal B}_{p}(\mu,L^2)$  is well defined and continuous.

\end{corollary}
\begin{proof}
As a consequence of \cite[Prop.5.1]{Matt87} and the comments above, we know that for every $0 \le \sigma < 1$,  $\Pi_1^\sigma(L^2,L^2= \Pi_2^\sigma(L^2,L^2)=\mathcal A_{\frac{2}{1-\sigma}}(L^2,L^2)$. Theorem \ref{sigma} gives the result.
\end{proof}

\subsection{Operators improving integrability of $\mathcal{L}^\infty$ and $C(K)$ valued functions}

Consider a space $L^\infty(\nu)$ for any measure $\nu$. By \cite[Prop.5.2]{Matt87}, for every Banach space $E$ and each  $1 \le p < \infty$, $0 \le \sigma < 1$ and $\varepsilon >0$, we have that
$$
\Pi_{\frac{p}{1-\sigma}}(L^\infty,E) \subseteq \Pi_{p}^\sigma(L^\infty,E)   \subseteq \Pi_{\frac{p}{1-\sigma} + \varepsilon}(L^\infty,E).
$$
This provides (see (1) below) a partial converse to Corollary \ref{c1}, which we reproduce for $L^\infty$-spaces in part (2) of the next result.

\begin{corollary}
Fix $1 \le q < p < \infty$, $\sigma=1/q'$, and consider a linear operator $u:L^\infty(\mu) \to E$.
\begin{itemize}
\item[(1)] If $u$
satisfies that its associated composition operator  $\tilde{u}$ is well-defined and continuous from ${\mathcal P}_q^\sigma(\mu,L^\infty)$ to ${\mathcal B}_{q}(\mu,E)$, then $u$ is $p$-summing.
\item[(2)]
In the case that $u$ is $q$-summing, then $\tilde{u}$ is well-defined and continuous from ${\mathcal P}_q^\sigma(\mu,L^\infty)$ to ${\mathcal B}_{q}(\mu,E)$
\end{itemize}
\end{corollary}
\begin{proof}
For the proof of (1), just note that by Theorem \ref{sigma}, $u$ belongs to  $\Pi_{1}^\sigma(L^\infty,E)$. The right hand side inclusion before the corollary gives (1). Item (2) is a particular case of Corollary \ref{c1}.
\end{proof}

As a consequence of a result of Pisier, it can be proved that the classes of $(\frac{1}{1-\sigma},1)$-summing operators and the class of $(1,\sigma)$-absolutely continuous operators coincide on $C(K)$-spaces (see \cite[Th.2.4 (ii)]{pisier} and the characterization of $(p,\sigma)$-absolutely continuous operators given by \cite[Th.4.1 (ii)]{Matt87} for the particular case of $C(K)$-spaces; see also \cite{LS97}). Recall that a consequence of Pisier's result is that for $1 \le q < p < \infty$, the $(p,q)$-summing operators acting in $C(K)$-spaces coincide with the $(p,1)$ summing ones. Thus, for a compact Hausdorff space $K$ and a Banach space $E$, we have the following

\begin{corollary}
Let $1 \le q < p < \infty$ and let $\sigma = 1/p'$. An operator $u:C(K) \to E$ is $(p,q)$-summing if and only if the composition operator $\tilde{u}$ from ${\mathcal P}_p^\sigma(\mu,C(K))$ to ${\mathcal B}_{p}(\mu,E)$ is well defined and continuous.
\end{corollary}

\subsection{Operators improving integrability of $\mathcal{L}^1$-spaces valued functions}

The case of $(p,\sigma)$-absolutely continuous operators from $\mathcal{L}^1$-spaces has been intensively studied (see \cite[\S 6]{Matt87}), so we can find a lot of applications of our results in this case. Let $ 0 \le \sigma <1$. Following the notation in this paper ---and taking into account the comments just after the definition in \cite[page 201]{Matt87}---, we define
$$
H_\sigma =:\big\{ E \, \textit{Banach}: \mathcal L(L^1,E) = \Pi_1^\sigma(L^1,E) \, \textit{for all $L^1$-spaces} \big\}.
$$

\begin{corollary}
Fix an $\mathcal{L}^1$-space $L^1$ and a non-atomic finite positive measure $\mu$.
Let $1 \le p < \infty$ and $\sigma=1/p'$. Consider a Banach space $E \in H_\sigma$.
For every operator  $u:L^1 \to E$, the composition operator $\tilde{u}$ from ${\mathcal P}_p^\sigma(\mu,L^1)$ to ${\mathcal B}_{p}(\mu,E)$ is well defined and continuous.
\end{corollary}

Let us apply this general result to some relevant Banach spaces belonging to the class $H_\sigma$. First, note that since the ideal $\Pi_1^\sigma$ is injective (\cite[Th.3.2.]{Matt87}) and $\ell^1$ has the lifting property all the subspaces of the spaces in $H_\sigma$ are again in $H_\sigma$. Following the notation of Matter in  \cite{Matt87}, we say that a Banach space $E$ is $(\sigma,p)$-Hilbertian ( $\sigma$-Hilbertian) if there is an interpolation pair $(H,E_1)$ with $H$ a Hilbert space and $E_1$ a Banach space such that $E=(H,E_1)_{\sigma,p}$ ($E=[H,E_1]_\sigma$) isomorphically. Here, $( \cdot,\cdot)_{\sigma,p}$ and $[\cdot,\cdot]_\sigma$ denote the usual real and complex interpolation spaces.

\begin{corollary}
Let $1 \le p < \infty$ and $\sigma=1/p'$. Fix an $\mathcal{L}^1$-space $L^1$ and a non-atomic finite positive measure $\mu$.
Each of the following Banach spaces $E$ satisfies that for every operator $u:L^1 \to E$, the associated composition operator  $\tilde{u}:{\mathcal P}_p^\sigma(\mu,L^1) \to {\mathcal B}_{p}(\mu,E)$ is well defined and continuous.
\begin{itemize}
\item[(1)] $E$ being a quotient of an $L^\infty$-space having cotype smaller that $\frac{2}{1-\sigma}.$

\item[(2)] $E= L^r$ for $2 \le r < \frac{2}{1-\sigma}$.

\item[(3)] $E$ being a $(\sigma,2)$-Hilbertian space.

\item[(4)] $E$ being a $\sigma$-Hilbertian space.

\item[(5)] $E$ being a Lorentz space $L_{r,s}$ for $\frac{2}{1+\sigma} < r$ and $s < \frac{2}{1-\sigma}$.

\end{itemize}

\end{corollary}
\begin{proof}
(1) Apply \cite[Prop.6.5]{Matt87}. (2) Apply \cite[Cor.6.6(a)]{Matt87}. (3) \cite[Th.8.1]{Matt87}. (4)  \cite[Th.8.2]{Matt87}. (5)  \cite[Th.9.2(a)]{Matt87}.
\end{proof}

Let us finish the paper with a result regarding super-reflexive Banach lattices. Recall that super-reflexive Banach spaces were the original motivation for introducing the ideal of $(p,\sigma)$-absolutely continuous operators. The following result is an application of our Theorem \ref{sigma} and Corollary 10.4 in \cite{Matt87}.

\begin{corollary}
Fix an $\mathcal{L}^1$-space $L^1$ and a non-atomic finite positive measure $\mu$.
If $F$ is a super-reflexive Banach lattice, then there is $1 \le p < \infty$
 such that for every operator $u:L^1 \to E$ the associated composition operator  $\tilde{u}:{\mathcal P}_p^\sigma(\mu,L^1) \to {\mathcal B}_{p}(\mu,F)$ is well defined and continuous,
for $\sigma=1/p'$.
\end{corollary}

{\bf Acknowledgments:} D. Pellegrino acknowledges with thanks the support of CNPq Grant 313797/2013-7 (Brazil). P. Rueda acknowledges with thanks the support of the Ministerio de
Econom\'{\i}a y Competitividad (Spain) MTM2011-22417. E.A. S\'anchez P\'erez
acknowledges with thanks the support of the Mi\-nisterio de Econom\'{\i}a y
Competitividad (Spain) MTM2012-36740-C02-02.

\end{document}